\documentclass{article}
\usepackage[utf8]{inputenc}
\usepackage{geometry}

\usepackage{parskip}
\usepackage{hyperref}
\usepackage{mathrsfs}  
\usepackage{amsmath}
\usepackage{amsthm}
\usepackage{amssymb}
\usepackage{mathtools}

\usepackage{complexity}
\usepackage{float}

\usepackage{tikz}
\usepackage{intcalc}
\usepackage{complexity}
\usepackage{subcaption}

\newcounter{counter1}
\newtheorem{lemma}[counter1]{Lemma}
\newtheorem{theorem}[counter1]{Theorem}
\newtheorem{corollary}[counter1]{Corollary}

\newtheorem{proposition}[counter1]{Proposition}
\theoremstyle{definition}
\newtheorem{remark}[counter1]{Remark}
\newtheorem{definition}[counter1]{Definition}
\usepackage[numbers]{natbib}

\DeclareMathOperator{\Basic}{Basic}
\DeclareMathOperator{\Cay}{Cay}
\DeclareMathOperator{\WL}{WL}
\DeclareMathOperator{\CR}{CR}

\newcommand{\Con}{U}
\newcommand{\con}{u}

\usepackage{hyperref}
\title{Combinatorial refinement on circulant graphs}
\author{\href{mailto:laurence.kluge@googlemail.com}{Laurence Kluge}\footnote{This work is based on the bachelor thesis of the author, which has emerged from the DFG project KO 1053/8--2.}\\ Institut für Informatik, Humboldt-Universität zu Berlin, Germany }
\date{}
\begin{document}
\newcommand{\high}{\textbf}

\maketitle
\vspace{1.3cm}
\paragraph{Abstract.} The combinatorial refinement techniques have proven to be an efficient
approach to isomorphism testing for particular classes of graphs.
If the number of refinement rounds is small, this puts
the corresponding isomorphism problem in a low-complexity class.
We investigate the round complexity of the 2-dimensional Weisfeiler-Leman algorithm
on circulant graphs, i.e.\ on Cayley graphs of the cyclic group $\mathbb{Z}_n$, and prove that the number of rounds until stabilization is bounded by $\mathcal{O}(d(n)\log n)$,
where $d(n)$ is the number of divisors of $n$. As a particular consequence, isomorphism can be tested in NC for connected circulant graphs of order $p^\ell$ with $p$ an odd prime, $\ell>3$ and vertex degree $\Delta$ smaller than $p$.

We also show that the color refinement method (also known as the 1-dimensional Weisfeiler-Leman algorithm)
computes a canonical labeling for every non-trivial circulant graph with a prime number of vertices
after individualization of two appropriately chosen vertices.
Thus, the canonical labeling problem for this class of graphs has at most the same
complexity as color refinement, which results in a time bound of $\mathcal{O}(\Delta \, n\log n)$.
Moreover, this provides a first example where 
a sophisticated approach to isomorphism testing put forward by Tinhofer has a real practical meaning.
\vfill
\tableofcontents
\vspace{0.3cm}
\newpage
\section{Introduction}
A classical algorithm to test for graph isomorphism is the $k$-dimensional Weisfeiler-Leman algorithm ($k$-WL), which was first suggested in its 2-dimensional form by Weisfeiler and Leman in \cite{WLe68} and investigated in general by Cai, Fürer and Immerman in \cite{CaiFI92}. Given a graph $G = (V, E)$ the $k$-WL algorithm computes a canonical coloring $\WL_k(G)$ of the set $V^k$ by repeated combinatorial refinement. Then two graphs $G_1$ and $G_2$ are decided as isomorphic if they have the same multiset of colors appearing in the coloring $\WL_{k}(G_i)$ for $i=1,2$. Isomorphic graphs are always recognized as isomorphic but for every $k \in \mathbb{N}$ there are examples of non-isomorphic graphs indistinguishable by $k$-WL. Nevertheless for many graph classes there is a $k$ such that the $k$-WL algorithm correctly decides graph isomorphism for all graphs of this class.

In particular Ponomarenko and Ryabov showed in \cite{ponomarenko2021pseudofrobenius} that $k=2$ is enough to distinguish Cayley graphs over the cyclic groups $\mathbb{Z}_n$ with Frobenius automorphism group and $n \not\in \{p, p^2, p^3, pq, p^2q\}$ for all primes $p$ and $q$. This includes all connected Cayley graphs over cyclic groups of order $p^\ell$ with $p$ an odd prime, $\ell>3$ and vertex degree $\Delta$ smaller than $p$.

We will primarily be interested in the $2$-WL algorithm. Unless otherwise stated, we always consider directed graphs without loops (which includes undirected graphs as a particular case). The \high{2-WL} algorithm starts with a coloring of $V^2$ into (up to) five colors given by the induced structure on the two vertices of a tuple (for each tuple $(v, w)$ the color depends on which of the following hold: $v =w$, $(v, w) \in E$ and $(w, v) \in E$). This coloring is then refined repeatedly until the partition stabilizes. Let $c(v,w)$ be the current color of the tuple $(v,w)$. In each step every tuple $(v_1, v_2)$ gets assigned the new color $\{\{(c(v_1, v_2), c(v, v_2), c(v_1, v)) \mid v \in V\}\}$ which is the multiset of colored triangles obtained by adding a third vertex $v \in V$ to the tuple $(v_1, v_2)$. To avoid exponential growth in the color names we would need to repeatedly give them new aliases. As a small example see Figure \ref{fig1}.

The number of rounds that the $k$-WL algorithm needs to distinguish two graphs is also important, as Grohe and Verbitsky showed in \cite{GroheV06} that the $r$-round $k$-WL algorithm can be implemented by TC circuits of depth $\mathcal{O}(r)$ and size $\mathcal{O}(r \cdot n^{3k})$. In particular, if the number of rounds is  polylogarithmic then graph isomorphism can be checked in \NC. For $2$-WL in general $\mathcal{O}(n \log n)$ rounds are always enough \cite{LichterPS19}, but there are examples where linearly many rounds are necessary \cite{Furer01}. For some special classes of graphs better upper bounds are known. For example, $2$-WL solves the isomorphism problem of trees in $\mathcal{O}(\log n)$ rounds \cite{PikhurkoV11}. With a larger dimension $k$, also the isomorphism problem of other graphs classes, such as graphs of bounded tree-width \cite{GroheV06} and planar graphs \cite{Verbitsky07,GroheK21} can be solved by $k$-WL in $\mathcal{O}(\log n)$ rounds.

Let $H$ be a finite group with neutral element $e$. For any subset $S \subseteq H \setminus \{e\}$, called the connection set, we form the \high{Cayley graph} $\Cay(H, S) = (V, E)$ with vertices $V = H$ and edges $(h,sh) \in E$ for any $h \in H$ and $s \in S$. A Cayley graph over the cyclic group $\mathbb{Z}_n$ is also called a circulant graph. We are going to interpret the $2$-WL algorithm on Cayley graphs as a refinement of specific linear subspaces of the group ring $\mathbb{Q}[H]$ that we call Schur-modules  (similar to the Schur-rings which are additionally closed under multiplication and therefore the results of the $2$-WL refinement). Based on this approach, we will prove:
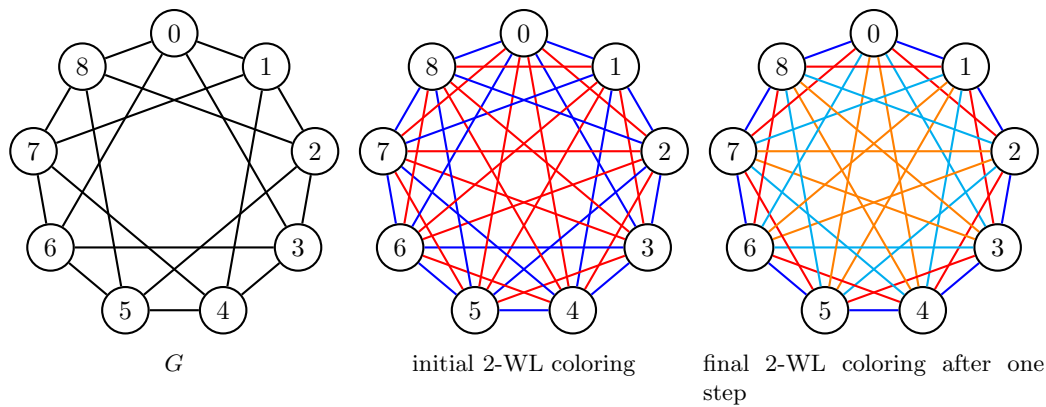
\begin{figure}
    \begin{subfigure}[t]{0.3 \textwidth}
        \centering
        \begin{tikzpicture}[node distance={15mm}, thick, main/.style = {draw, circle}] 
            \foreach \n in {0,...,8} {
                \node[main] (\n) at ({1.9*sin((360*\n) / 9)}, {1.9*cos((360*\n)/9)}) {\n};
            }
            \foreach \n in {0,...,8} {
                \foreach \m in {1, 3} {
                    \pgfmathtruncatemacro{\cur}{\intcalcMod{\n+\m}{9}}
                     \draw[-] (\n) -- (\cur);
                }
            }
        \end{tikzpicture}
        \caption*{$G$}
    \end{subfigure}
    \begin{subfigure}[t]{0.3 \textwidth}
        \centering
        \begin{tikzpicture}[node distance={15mm}, thick, main/.style = {draw, circle}] 
            \foreach \n in {0,...,8} {
                \node[main] (\n) at ({1.9*sin((360*\n) / 9)}, {1.9*cos((360*\n)/9)}) {\n};
            }
            \foreach \n in {0,...,8} {
                \foreach \m in {1, 3} {
                    \pgfmathtruncatemacro{\cur}{\intcalcMod{\n+\m}{9}}
                    \draw[-, draw=blue] (\n) -- (\cur);
                }
                \foreach \m in {2, 4} {
                    \pgfmathtruncatemacro{\cur}{\intcalcMod{\n+\m}{9}}
                    \draw[-, draw=red] (\n) -- (\cur);
                }
            }
        \end{tikzpicture}
        \caption*{initial $2$-WL coloring}
    \end{subfigure}
    \begin{subfigure}[t]{0.3 \textwidth}
        \centering
        \begin{tikzpicture}[node distance={15mm}, thick, main/.style = {draw, circle}] 
            \foreach \n in {0,...,8} {
                \node[main] (\n) at ({1.9*sin((360*\n) / 9)}, {1.9*cos((360*\n)/9)}) {\n};
            }
            \foreach \n in {0,...,8} {
                \foreach \m in {1} {
                  \pgfmathtruncatemacro{\cur}{\intcalcMod{\n+\m}{9}}
                    \draw[-, draw=blue] (\n) -- (\cur);
                }
                \foreach \m in {3} {
                  \pgfmathtruncatemacro{\cur}{\intcalcMod{\n+\m}{9}}
                    \draw[-, draw=cyan] (\n) -- (\cur);
                }
                \foreach \m in {2} {
                  \pgfmathtruncatemacro{\cur}{\intcalcMod{\n+\m}{9}}
                    \draw[-, draw=red] (\n) -- (\cur);
                }
                \foreach \m in {4} {
                  \pgfmathtruncatemacro{\cur}{\intcalcMod{\n+\m}{9}}
                    \draw[-, draw=orange] (\n) -- (\cur);
                }
            }
        \end{tikzpicture}
        \caption*{final $2$-WL coloring after one step}
    \end{subfigure}
    \caption{The $2$-WL algorithm on the Cayley graph $G=\Cay(\mathbb{Z}_9, \{1,3, 6, 8\})$.} \label{fig1}
\end{figure} %
\begin{theorem} \label{dnlogn_bound}
For all $n \in \mathbb{N}$ and all circulant graphs of order $n$ the 2-dimensional Weisfeiler-Leman algorithm terminates in $\mathcal{O}(d(n) \log(n))$ steps, where $d(n)$ is the number of divisors of $n$.
\end{theorem}
We say that a property holds for almost all $n$, if the proportion of values $n \leq N$ with this property tends to $1$ as $N$ tends to infinity. Let $\epsilon > 0$ be fixed. As $\sum_{i=1}^n d(i) = \mathcal{O}(n\log n)$, we also get $d(n) \leq \log(n)^{1+\epsilon}$ for almost all $n$ and then Theorem \ref{dnlogn_bound} implies:
\begin{corollary}
For almost all $n \in \mathbb{N}$ the 2-dimensional Weisfeiler-Leman algorithm terminates on all circulant graph of order $n$  in $\mathcal{O}(\log(n)^{2+\epsilon})$ steps. 
\end{corollary} 
By using the aforementioned results of \cite{GroheV06} and \cite{ponomarenko2021pseudofrobenius} we also get:
\begin{corollary}
For connected circulant graphs of order $p^\ell$ for $p$ prime, $\ell > 3$ and vertex degree $\Delta$ smaller than $p$, isomorphism can be tested in \NC.
\end{corollary}
The 1-dimensional version of $k$-WL is commonly known as the \high{Color-Refinement algorithm}. Given a coloring $c$ of the vertices of a graph $G = (V, E)$, the coloring is refined repeatedly by giving each vertex $v$ a new color depending on the label $\{\{c(h) \mid h \in V \text{ with } (h, v) \in E \}\}$ i.e.\ the multiset of colors of vertices with an edge to $v$. This step is repeated until the partition of vertices stabilizes. Note that we define Color-Refinement for directed graphs by only using the in-neighbours. In the case of undirected graphs this coincides with the standard definition.

If we start with a uniform coloring of the vertices this could never be sufficient to test isomorphism with Cayley graphs, as regular graphs are not refined at all and hence regular graphs of the same order and degree cannot be distinguished. Tinhofer used the Color-Refinement to give an algorithm that behaves inversely to the $k$-WL approach: It always distinguishes non-isomorphic graphs but may also classify isomorphic graphs as non-isomorphic. For two graphs $G$ and $H$ the algorithm performs the following steps (see e.g.\ \cite[Section 7]{ArvindKRV17}):
\begin{itemize}
    \item[1.] Give every vertex of the disjoint union of $G$ and $H$ the same color.
    \item[2.] Run the Color Refinement algorithm on the current coloring of $V(G) \cup V(H)$ until it stabilizes.
    \item[3.] If the multisets of colors in $G$ and $H$ are different, decide that $G$ and $H$ are not isomorphic.
    \item[4.] If all vertices of $G$ or equivalently of $H$ are colored uniquely decide that $G$ and $H$ are isomorphic.
    \item[5.] Choose a color class with at least two vertices in both $G$ and $H$, select $v \in V(G)$ and $w \in V(H)$ in this class and give them the same new unique color. Repeat from Step 2.
\end{itemize}
We say that a graph $G$ has the \high{Tinhofer property} if the algorithm above is correct for every $H$ and every choice of vertices to be individualized. Tinhofer showed that this approach works for compact graphs \cite{Tinhofer91}, but Schreck and Tinhofer have also shown in \cite{SchreckT88} that circulant graphs of prime order are rarely compact.

We apply the tools of linear algebra to analyse the outcome of the Color-Refinement algorithm on circulant graphs of prime order after individualizing some vertices. We use this to prove the following theorem, which answers a question posed by Arvind et al. \cite{ArvindKRV17} and gives a first example that Tinhofer's approach to isomorphism testing really works for a non-trivial natural class of graphs: 
\begin{theorem} \label{tinhofer_prop}
Every circulant graph of prime order has the Tinhofer property.
\end{theorem}
In particular we show that for non-trivial circulant graphs of prime order at most $2$ vertices need to be individualized until every vertex gets a unique color. The Color Refinement algorithm can be implemented in time $\mathcal{O}(\Delta\, p\log p)$ \cite{Cardon1982PartitioningAG}, where $\Delta$ is the vertex degree. This implies that we can test isomorphism with the same time bound. Moreover, the Tinhofer property implies an algorithm to compute a canonical labeling \cite[Lemma 7.1]{ArvindKRV17}. Therefore, we get an efficient canconical labeling for circulant graphs of prime order:
\begin{corollary}
Circulant graphs of prime order $p$ and vertex degree $\Delta$ admit a canonical labeling computable in time $\mathcal{O}(\Delta \, p \log p)$.
\end{corollary}
Finally, we present an undirected Cayley graph over $\mathbb{Z}_4 \times \mathbb{Z}_4$ which does not have the Tinhofer property, showing that not every (undirected) Cayley graph over an (abelian) group has the Tinhofer property.
\paragraph{Related work.} The isomorphism problem for the class of all circulant graphs is solvable in polynomial time due to Evdokimov and Ponomarenko \cite{EvdokimovP04} and independently Muzychuk \cite{muzychuk2004solution} (however the setting of the latter is that the graphs are given by a connection set of integers instead of by an adjacency matrix). Whether or not this can be done by $k$-WL remains widely open. In \cite{EvdokimovP02} Evdokimov and Ponomarenko investigate \emph{normal} circulant graphs, whose automorphisms are induced by automorphisms of the underlying cyclic group. Since non-trivial circulant graphs of prime order are normal, their analysis implies an analog of our Theorem \ref{tinhofer_prop} for the weaker version of Tinhofer's property where Color Refinement is replaced with $2$-WL.

\paragraph{Organization of the paper.} In Section \ref{2} we give basic algebraic definitions needed for our analysis of 2-WL and Color-Refinement. Then we algebraically interpret $2$-WL on Cayley graphs in Section \ref{3.1} and use this to prove Theorem \ref{wl_steps_cyclic_bound} in Section \ref{3.2}, which is a slightly more general variant of Theorem \ref{dnlogn_bound}. Finally, in Section \ref{4} we investigate the Color-Refinement, prove Theorem \ref{tinhofer_prop} and give an example of a Cayley graph over $\mathbb{Z}_4 \times \mathbb{Z}_4$ that does not have the Tinhofer property.

\paragraph{Acknowledgement.}
I would like to thank Dr.\ Oleg Verbitsky for a lot of helpful comments and especially for providing the context for the results of this paper.

\section{S-modules and partitions} \label{2}
Let $K$ be a  field and $G$ a finite group. Let $e \in G$ be the neutral element.
The \high{group ring} $K[G]$ is defined to be the vector space over $K$ with basis elements $\underline{g}$ for $g \in G$ (i.e.\ we write elements of $K[G]$ as $\sum_{g\in G} \lambda_g \underline{g}$ with all $\lambda_g \in K$), endowed with the multiplication given by:
$$(\sum_{g \in G}\lambda_g \underline{g}) \cdot (\sum_{g \in G} \mu_g \underline{g}) := \sum_{g_1, g_2 \in G} \lambda_{g_1} \mu_{g_2} \underline{g_1g_2}$$
This is a ring with unit $\underline{e} \in K[G]$ and it is commutative precisely if $G$ is abelian.
\begin{definition}
For a subset $T \subseteq G$ we define the \high{simple quantity} $\underline{T} := \sum_{g \in T} \underline{g} $. \\
For a map $\phi : G \rightarrow G$ we define
$$(\sum_{g\in G} \lambda_g \underline{g})^\phi := \sum_{g\in G} \lambda_g \underline{\phi(g)}$$
In particular, for integers $m$ we have the map $G \rightarrow G$: $g \mapsto g^m$ (which is bijective for $m$ coprime to $|G|$) and for this we also write:
$$(\sum_{g\in G} \lambda_g \underline{g})^{(m)} := \sum_{g\in G} \lambda_g \underline{g^m}$$
Similarly we define for any $T \subseteq G$: $T^\phi := \{\phi(g) \mid g \in T\}$ and $T^{(m)} := \{g^m\mid g \in T\}$. 
\end{definition}
In the following we will define S-modules and prove a few basic properties also found in Sections 22 and 23 of \cite{Wie}.
\begin{definition}
We call a linear subspace $V \subseteq K[G]$ a \high{Schur-module} (or \high{S-module}) if there is some partition $\mathcal{C} = \{C_1, \dots C_r\}$ of $G$ (i.e.\ $\bigcup_{1 \leq i \leq r} C_i = G$ and the $C_i$ are pairwise disjoint and non-empty) such that $\{\underline{C_i}\}_{1\leq i \leq r}$ is a basis of $V$. We call the sets $C_i$ the \high{basic sets} of $V$, and as the partition is unique we denote $\Basic(V) := \mathcal{C}$. Every $g\in G$ is contained in exactly one set $C_i$ and we write $C_{(g)} := C_i$. \\
If further $\underline{e} \in V$, $V$ is closed under multiplication (induced by $K[G]$) and for any $C \in Basic(V)$ we have $C^{(-1)} \in Basic(V)$, then $V$ is known as a \high{Schur-Ring} (or \high{S-Ring}).
\end{definition}
The S-modules are basically equivalent to partitions of $G$, but they provide a natural product operation that incorporates the group structure of $G$ and they are linearized, which allows to use the tools of linear algebra.\\  
A basic property of S-modules is that elements of $G$ which are in the same basic set have the same coefficient in any element of the S-module. This allows us to extract and combine elements:
\begin{lemma} \label{S_module_extract}
Let $V$ be an S-module and let $v_i = \sum_{g \in G} \lambda_{i,g} \underline{g}$ be elements of $V$ for $1 \leq i \leq m$. Then for any map $\phi : K^m \rightarrow K$ also $\sum_{g \in G} \phi(\lambda_{1,g}, \dots, \lambda_{m,g}) \underline{g}$ is an element of $V$.
\end{lemma}
\begin{proof}
Let $C_1, \dots, C_r$ be the basic sets of $V$. For every $1 \leq i \leq m$  $v_i$ is an element of $V$, so we can write $v_i = \sum_{1\leq j \leq r} \mu_{i,j} \underline{C_j}$ for some $\mu_{i,j} \in K$. Given any $g \in C_j$ comparing coefficients of $\underline{g}$ in $v_i$ gives $\lambda_{i,g} = \mu_{i,j}$, as $\underline{g}$ only appears in $\underline{C_j}$. But then $\phi(\lambda_{1,g}, \dots, \lambda_{m,g}) = \phi(\mu_{1,j}, \dots, \mu_{m,j})$ implies $\sum_{g\in G}\phi(\lambda_{1,g}, \dots, \lambda_{m,g})\underline{g} =\sum_{1\leq j \leq r}\phi(\mu_{1,j}, \dots, \mu_{m,j}) \underline{C_j} \in V$ as desired.
\end{proof}
In particular this allows us to extract the parts of $v \in V$ with different coefficients:
\begin{corollary} \label{extract_coefficients}
Let $V$ be an S-module, $v = \sum_{g \in G} \lambda_g \underline{g}$ an element of $V$ and $\lambda \in K$. If $T$ is the set of all $g$ which appear with coefficient $\lambda$ in $v$, then also $\underline{T} \in V$.
\end{corollary}
\begin{proof}
This follows immediately by applying Lemma \ref{S_module_extract} to $v$ with $\phi(x) = \begin{cases}
1 & x = \lambda \\
0 & x \neq \lambda
\end{cases}$.
\end{proof}
And we can also take intersections of simple quantities in an S-module:
\begin{corollary} \label{s_module_intersections}
Let $V$ be an S-module and $T_1, T_2$ subsets of $G$ with $\underline{T_1}, \underline{T_2} \in V$. Then also $\underline{T_1 \cap T_2} \in V$ and $\underline{T_1 \cup T_2} \in V$.
\end{corollary}
\begin{proof}
This follows immediately by applying Lemma \ref{S_module_extract} to $\underline{T_1}$ and $\underline{T_2}$ with the maps $\phi(x,y) = \begin{cases}
1 & x = 1 \wedge y = 1 \\
0 & \text{otherwise}
\end{cases}$ and $\phi(x,y) = \begin{cases}
1 & x = 1 \vee y = 1 \\
0 & \text{otherwise}
\end{cases}$ respectively.
\end{proof}
\begin{definition}
The \high{meet} of two partitions $\mathcal{P} = \{P_1, \dots, P_n\}$ and $\mathcal{P}'= \{P_1', \dots, P_m'\}$ is the coarsest common subpartition given by $\mathcal{P} \wedge \mathcal{P}' = \{P_i \cap P_j' \mid 1\leq i \leq n, 1\leq j \leq m\}$.
\end{definition}
\begin{definition}
Given a partition $\mathcal{P} = \{P_1, \dots, P_n\}$ of a set $X$ and $x\in X$, we also write $\mathcal{P}(x) := P_i$ for the unique set $P_i$ with $x \in P_i$ and $1\leq i \leq n$.
\end{definition}
\section{2-WL on circulant graphs} 
\subsection{General constructions}\label{3.1}
In this section we interpret the 2-WL algorithm on Cayley graphs of finite groups in terms of operations on S-modules. For this $G$ is still an arbitrary finite group, but we pick $K= \mathbb{Q}$ (the important part is that we pick a field of characteristic $0$, as otherwise we would loose information about the coefficients in the products of basic quantities in $K[G]$).\\
The central notion in the context of the 2-WL algorithm is that of a refinement, which is going to correspond to a single refinement step of the 2-WL algorithm:
\begin{definition}For any element $v = \sum_{g \in G}\lambda_g \underline{g} \in \mathbb{Q}[G]$ we define the induced partition on $G$ given by the elements of $G$ with identical coefficients $\mathcal{C}_v := \{\{g \in G \mid \lambda_g = \lambda\} \mid \lambda \in \mathbb{Q}\}$.\\
Let $V$ be an S-module with basic sets $\Basic(V) = \{T_1, \dots, T_r\}$ where the $T_i$ are pairwise distinct. We define the \high{refinement} of $V$ to be the S-module $R(V)$ corresponding to the partition $\Basic(V) \wedge \bigwedge_{1 \leq i, j \leq r} \mathcal{C}_{\underline{T_i} \cdot \underline{T_j}}$.
\end{definition}
The refinement of an S-module is the smallest S-module that contains all products of two elements (and as it is an S-module, also all elements derived from Corollary \ref{extract_coefficients} and \ref{s_module_intersections}):
\begin{lemma} \label{refinement_product}
Let $V$ be an S-module and $v, w$ elements of $V$. Then the product $v \cdot w$ in $\mathbb{Q}[G]$ is an element of $R(V)$.
\end{lemma}
\begin{proof}
Let $\Basic(V) = \{T_1, \dots, T_r\}$. By definition $\underline{T_i} \cdot \underline{T_j} \in R(V)$ for every $1 \leq i,j \leq r$. The $\underline{T_i}$ form a basis of $V$ so we can write $v = \sum_{1\leq i \leq r}\mu_i \underline{T_i}$ and $w = \sum_{1\leq i \leq r}\mu_i' \underline{T_i}$ with $\mu_i, \mu_i' \in \mathbb{Q}$. Then $v \cdot w = \sum_{1\leq i,j \leq r} \mu_i\mu_j (\underline{T_i}\cdot \underline{T_j}) \in R(V)$ as $R(V)$ is a vector space.
\end{proof}
\begin{lemma} \label{subset_refinement}
Let $V, W$ be S-modules such that $V \subseteq W$. Then also $R(V) \subseteq R(W)$.
\end{lemma}
\begin{proof}
Let $\Basic(V) = \{T_1, \dots, T_r\}$. By definition an element $T \in \Basic(R(V))$ is of the form $T = T_\ell \cap \bigcap_{1\leq i,j \leq r} T_{i,j}$ for some $1\leq \ell \leq r$ and where the $T_{i,j}$ are the elements of $G$ which appear in $\underline{T_i}\cdot \underline{T_j}$ with coefficient $\lambda_{i,j}\in \mathbb{Q}$. We have $\underline{T_i} \in V \subseteq W$ and thus also $\underline{T_i}\cdot \underline{T_j} \in R(W)$ by Lemma \ref{refinement_product}. Applying Lemma \ref{S_module_extract} in $R(W)$ to $\underline{T_\ell}, \underline{T_{1}}\cdot \underline{T_1}, \underline{T_{1}}\cdot \underline{T_2}, \dots  ,\underline{T_r}\cdot \underline{T_r}$ with $$\phi(x,x_{1,1}, x_{1,2}, \dots , x_{r,r}) = \begin{cases}
1 & x = 1 \wedge \bigwedge_{1\leq i, j \leq r} x_{i,j} = \lambda_{i,j} \\
0 & \text{otherwise} 
\end{cases}$$ implies $\underline{T} \in \Basic(R(W))$. Hence $R(V) \subseteq R(W)$.
\end{proof}
The 2-WL algorithm acts on partitions of pairs of vertices, which corresponds to $G\times G$ in the case of Cayley graphs. The partitions that will appear are going to satisfy some extra properties, which allow us to go back to a partition of $G$ and hence to an S-module:
\begin{definition}
Let $\mathcal{C}$ be a partition of $G \times G$. We say that $\mathcal{C}$ is a \high{Cayley partition} if the following hold:
\begin{itemize}
    \item for every $P \in \mathcal{C}$ we have that $(g_1, g_2) \in P$ if and only if $(g_1g, g_2g) \in P$ for all $g \in G$,
    \item there is some $P \in \mathcal{C}$ with $ P = \{(g, g) \mid g \in G\}$,
    \item for every $P \in \mathcal{C}$ there is some $P' \in \mathcal{C}$ with $P' = \{(g_2, g_1) \mid (g_1, g_2) \in P\}$.
\end{itemize}
\end{definition}
\begin{definition}
Let $\mathcal{C} = \{P_1, \dots, P_r\}$ be any partition of $G \times G$.  
For every $(g_1, g_2) \in G\times G$ we consider the new label $L(g_1, g_2)$ given by the multiset of colors of triangles (where as the color of an edge we understand the set of $\mathcal{C}$ an edge is in) we obtain from $(g_1, g_2)$ by adding a third vertex $g$:   $$L(g_1, g_2) := (\mathcal{C}((g_1, g_2)), \{\{(\mathcal{C}((g_1, g)), \mathcal{C}((g, g_2))) \mid g \in G \}\})$$ One step of the \high{2-WL algorithm} on $\mathcal{C}$ (where we interpret $\mathcal{C}$ as the tuple color classes of a graph with vertices $G$) is defined to be the partition $\WL_2(\mathcal{C})$ induced by splitting elements with different labels $L$. The sets of $\WL_2(\mathcal{C})$ are then given by $\{(g_1', g_2') \in G\times G \mid L(g_1,g_2) = L(g_1', g_2')\}$ for $(g_1, g_2) \in G\times G$.
\end{definition}
\begin{lemma}
Let $\phi : G \rightarrow G$ be a bijective map. If $\mathcal{C} = \{P_1, \dots, P_r\}$ is a partition of $G\times G$ that satisfies $(g_1, g_2) \in P_i \Leftrightarrow (\phi(g_1), \phi(g_2)) \in P_i$ for all $g_1, g_2 \in G$ and $1\leq i \leq r$ then also $\WL_2(\mathcal{C})$ satisfies this. In particular, if $\mathcal{C}$ is a Cayley partition then so is $\WL_2(\mathcal{C})$.
\end{lemma}
\begin{proof}
The hypothesis says $\mathcal{C}((g_1, g_2)) = \mathcal{C}((\phi(g_1), \phi(g_2)))$ for any $(g_1, g_2) \in G \times G$. Then we compute for any $(g_1, g_2) \in G\times G$:
\begin{align*}
    \hspace{-7.5pt}L(\phi(g_1), \phi(g_2)) &= (\mathcal{C}((\phi(g_1), \phi(g_2))), \{\{(\mathcal{C}((\phi(g_1), g)), \mathcal{C}((g, \phi(g_2)))) \mid g \in G \}\}) \\
    &= (\mathcal{C}((\phi(g_1), \phi(g_2))), \{\{(\mathcal{C}((\phi(g_1), \phi(g))), \mathcal{C}((\phi(g), \phi(g_2)))) \mid g \in G \}\}) \\
    &= (\mathcal{C}((g_1, g_2)), \{\{(\mathcal{C}((g_1, g)), \mathcal{C}((g, g_2))) \mid g \in G \}\}) \\
    &= L(g_1, g_2)
\end{align*}
This immediately implies $\WL_2(\mathcal{C})((\phi(g_1),\phi(g_2)) = \WL_2(\mathcal{C})((g_1, g_2))$, i.e.\ the desired property. 

Finally, for a fixed $g\in G$ the map $\phi : G\rightarrow G$: $g' \mapsto g'g$ is bijective and hence what we just showed proves that the first two properties of being a Cayley partition are fulfilled. The third property follows similarly by observing that if $L(g_1, g_2) = L(g_1', g_2')$ then also $L(g_2, g_1) = L(g_2', g_1')$.
\end{proof}
For a Cayley partition we have a lot of redundant information in any of the sets. Therefore we can pick out only one of the the equivalent elements $(g_1g, g_2g)$ in every set and thus reduce to a partition of $G$, which we interpret as an S-module: 
\begin{definition}
To every Cayley partition $\mathcal{C} = \{P_1, \dots, P_n\}$ of $G \times G$ we associate the \high{induced S-module} $S(\mathcal{C})$ given by the basic sets $T_i := \{ a \in G \mid (e, a) \in P_i\}$. 
\end{definition}
In this context the refinement of S-modules corresponds to the steps of the $2$-WL algorithm, similar to how S-rings correspond to the output of the $2$-WL algorithm:
\begin{lemma} \label{wl_refinement}
Let $\mathcal{C}$ be a Cayley partition of $G \times G$. Then $S(\WL_2(\mathcal{C})) = R(S(\mathcal{C}))$ i.e.\ the refinement of the induced S-module is precisely the induced S-module after applying the $2$-WL algorithm once.
\end{lemma}
\begin{proof}
Write $\mathcal{C} = \{P_1, \dots, P_r\}$. Then the $T_i := \{g \in G \mid (e,g) \in P_i\}$ form the basic sets of $S(\mathcal{C})$. For $g \in G$ and $1 \leq i, j \leq r$ we observe the following:
\begin{align*}
\text{coefficient of } \underline{g} \text{ in } \underline{T_i}\cdot \underline{T_j} 
&= | \{g_1, g_2 \in G \mid g_1 \in T_i, g_2 \in T_j \wedge g_1 \cdot g_2 = g\} |  \\
&= | \{g_1, g_2 \in G \mid (e,g_1) \in P_i, (e,g_2) \in P_j \wedge g_1 \cdot g_2 = g\} |  \\
&= | \{g_1, g_2 \in G \mid (g_2,g) \in P_i, (e, g_2) \in P_j \wedge g_1 \cdot g_2 = g\} |  \\
&= | \{g_2 \in G \mid (g_2, g) \in P_i, (e, g_2) \in P_j\} |  \\
&= \text{multiplicity of } (P_j, P_i) \text{ in the multiset of } L((e, g))
\end{align*}
where we used the Cayley property for the third equality and for the fourth equality we used that $g_2$ is uniquely determined by $g_1$.

Let $T \in \Basic(S(\WL_2(\mathcal{C})))$. This is of the form $T = \{g \in G \mid (e, g) \in P\}$ for an element $P\in \WL_2(\mathcal{C})$, i.e.\ $P = \{(g_1', g_2') \in G \times G  \mid L(g_1', g_2') = L(g_1, g_2)\}$ for some $(g_1, g_2) \in G\times G$.  Let $\lambda_{i,j}$ be the multiplicity of $(P_i, P_j)$ in $L(g_1, g_2)$ for all $1\leq i,j \leq r$. Then we have:
\begin{align*}
    \hspace{-1pt}T 
    &= \{g \in G \mid L(e, g) = L(g_1, g_2)\} \\
    &= \{g \in G \mid \mathcal{C}((e,g))=\mathcal{C}((g_1,g_2)) \text{ and the multiplicity of } ( P_i, P_j)  \text{ in } L(e, g) \\ &\hspace{2.03cm} \text{ is } \lambda_{i, j} \text{ for all } i,j\} \\
    &= \{g \in G \mid \mathcal{C}((e,g))=\mathcal{C}((g_1,g_2)) \text{ and coefficient of } \underline{g} \text{ in } \underline{T_j}\cdot\underline{T_i} \text{ is } \lambda_{i, j} \text{ for all } i,j\} \\
    &= T_{(g_1,g_2)} \cap \bigcap_{1\leq i,j \leq r} \{g \in G \mid \text{ the coefficient of } g \text{ in } \underline{T_j}\cdot \underline{T_i} \text{ is } \lambda_{i,j}\}
\end{align*}
where $T_{(g_1,g_2)}$ is the basic set of $S(\mathcal{C})$ that contains $(e,g_2g_1^{-1})$ (i.e.\ which corresponds to the $P = \mathcal{C}((g_1,g_2)) \in \mathcal{C}$ with $(g_1,g_2) \in P$). But this is precisely a basic set of $R(S(\mathcal{C}))$.

This shows $\Basic(S(\WL_2(\mathcal{C}))) \subseteq \Basic(P(S(\mathcal{C})))$. But both sets are partitions of $G$, and thus they have to be equal, i.e.\ $\Basic(S(\WL_2(\mathcal{C}))) = \Basic(P(S(\mathcal{C})))$ as desired.
\end{proof}
Now if we form the Cayley graph $\Cay(G, S) = (G, E)$ for some subset $S \subseteq G$ (with $E$ the edges of this graph), then the 2-WL algorithm associates to it the initial coloring of tuples based on the edge properties, which then corresponds to the partition: 
\begin{align*}
    \mathcal{C} := \{&\{(g,g) \mid g \in G\}, \\ &\{(g_1,g_2) \in G\times G \mid g_1\neq g_2 \wedge (g_1, g_2) \not\in E \wedge (g_2, g_1) \not\in E\}, \\
    &\{(g_1,g_2) \in G\times G \mid g_1\neq g_2 \wedge (g_1, g_2) \in E \wedge (g_2, g_1) \not\in E\}, \\
    &\{(g_1,g_2) \in G\times G \mid g_1\neq g_2 \wedge (g_1, g_2) \not\in E \wedge (g_2, g_1) \in E\}, \\
    &\{(g_1,g_2) \in G\times G \mid g_1\neq g_2 \wedge (g_1, g_2) \in E \wedge (g_2, g_1) \in E\}\}
\end{align*}
As this is a Cayley graph we have $(g_1, g_2) \in E \Leftrightarrow (g_1g, g_2g) \in E$, i.e.\ $\mathcal{C}$ is a Cayley partition. \\
But Lemma \ref{wl_refinement} implies inductively that $S(\WL_2^i(\mathcal{C})) = R^i(S(\mathcal{C}))$ holds for all $i \in \mathbb{N}$. The process of turning a Cayley partition into an S-module is injective (we can reconstruct the Cayley partition from the S-module by reintroducing all the equivalent elements) and hence we get:
\begin{proposition}\label{equivalence_wl_refinement} The 2-WL algorithm stabilizes the Cayley partition $\mathcal{C}$ after the same amount of steps as the refinement of S-modules does on the induced S-module $S(\mathcal{C})$.
\end{proposition} 
This allows us to analyse the refinement procedure to prove bounds on the amount of steps of the 2-WL algorithm on Cayley graphs. 

\subsection{Steps of 2-WL} \label{3.2}
In this section we start the analysis of the refinement procedure. From now on we restrict G to be an abelian group (but we will continue writing the group multiplicatively), which implies that $\mathbb{Q}[G]$ is a commutative ring. We will use the following concept of exponentiation-stable S-modules to prove bounds on the amount of refinement steps. 
\begin{definition}
We call an S-module $V$ \high{exponentiation-stable}, if for every $T \subseteq G$ with $\underline{T} \in V$ and all positive integers $m$ which are coprime to $|G|$ we also have that $\underline{T}^{(m)} \in V$.
\end{definition}
The Schur theorem on multipliers states that S-rings are exponentiation-stable (more precisely, see \cite[Theorem 2.4.10]{CP2019}) and similarly we prove that after $\mathcal{O}(\log(|G|))$ refinement steps we can assume our  S-module to be exponentiation-stable:
\begin{lemma} \label{cylic_log_n}
Let $V$ be an S-module, $T \subseteq G$ with $\underline{T} \in V$ and $m$ a positive integer that is coprime to $|G|$. Then $\underline{T}^{(m)} \in R^{r}(V)$ with $r = 2 \lceil \log{m} \rceil$.
\end{lemma}
\begin{proof}
We start with the case $m = p$ prime:

Similar to the technique of \textit{exponentiation by squaring} we inductively get $\underline{T}^j \in R^i(V)$ for $1 \leq j \leq 2^i$ and all $i \in \mathbb{N}$. The base case $i = 0$ is trivial. In the general case for $i \geq 1$ and $1\leq j \leq 2^i$  we write $j = j_1 + j_2$ with $0 \leq j_1, j_2 \leq 2^{i-1}$ and get $\underline{T}^j = \underline{T}^{j_1} \cdot \underline{T}^{j_2} \in R(R^{i-1}(V)) = R^i(V)$ by induction and Lemma \ref{refinement_product}. In particular we have $\underline{T}^p \in R^r(V)$ with $r = \lceil \log(p) \rceil$.

Now we use the binomial theorem (for which we need that $\mathbb{Q}[G]$ is a commutative ring) to express $\underline{T}^p$: Inductively on the size of any subset $P \subseteq G$ we show that there exist $\lambda_{P,g} \in \mathbb{Z}$ for $g \in G$ such that $$\underline{P}^p = \sum_{g \in P} \underline{g^p} + p \cdot \sum_{g \in G} \lambda_{P, g} \underline{g}$$
The base case $|P| = 0$ is trivial. Let $|P| \geq 1$ and $g \in P$. We compute:
\begin{align*}
    \underline{P}^p &= (\underline{g} + \underline{P\setminus \{g\}})^p = \sum_{k = 0}^{p} \binom{p}{k}\underline{g}^k\underline{P\setminus \{g\}}^{p-k} \\ &= \underline{g}^p + \underline{P\setminus \{g\}}^p  + \sum_{k = 1}^{p-1} \binom{p}{k}\underline{g}^k\underline{P\setminus \{g\}}^{p-k} \\ &= \underline{g^p} + \sum_{g' \in P\setminus\{g\}}\underline{g'^p} + \sum_{k = 1}^{p-1} \binom{p}{k}\underline{g}^k\underline{P\setminus \{g\}}^{p-k} + p \cdot \sum_{g' \in G} \lambda_{P \setminus\{g\}, g'} \underline{g'} \\ &= \sum_{g' \in P}\underline{g'^p} + p \cdot \sum_{g' \in G} \lambda_{P, g'} \underline{g'}
\end{align*}
for some $\lambda_{P, g} \in \mathbb{Z}$, as $\binom{p}{k}$ is divisible by $p$ for all $1 \leq k \leq p-1$ and the products $\underline{g}^k\underline{P\setminus \{g\}}^{p-k}$ are again $\mathbb{Z}$-linear combinations of the $\underline{g'}$. \\
In particular we can now write:
$$\underline{T}^p = \sum_{g \in T} \underline{g^p} + p \cdot \sum_{g \in G} \lambda_{T, g} \underline{g} $$
As $p$ is coprime to $|G|$ the map $G \rightarrow G$: $g \mapsto g^p$ is injective and thus every $g^p$ for $g \in T$ is unique, i.e.\ only the $g' \in T^{(p)}$ appear with a coefficient not divisible by $p$ in the above expression.\\
But this implies that $\underline{T}^{(p)}$ is the element we obtain if we apply Lemma \ref{S_module_extract} in $R^r(V)$ to $\underline{T}^p$ with $\phi(x) = \begin{cases}
1 & x \not\in p \mathbb{Z} \\
0 & x \in p \mathbb{Z}
\end{cases}$ and thus $\underline{T}^{(p)} \in R^r(V)$ for $r = \lceil \log(p) \rceil$.

In general we write $m = p_1^{a_1} \cdot \dots \cdot p_\ell^{a_\ell}$ the prime factorization of $m$.
Then $\underline{T}^{(m)} = (((\underline{T}^{(p_1)})^{(p_1)}) \dots )^{(p_\ell)}$. All the $p_i$ are also coprime to $|G|$ and by repeatedly applying the prime case we get $T^{(m)} \in R^r(V)$ for $r = a_1 \lceil \log(p_1) \rceil + \dots + a_\ell \lceil \log(p_\ell) \rceil \leq 2 \lceil \log(m) \rceil$ as desired.
\end{proof}
Finally we need to see that this property is stable under refinement:
\begin{lemma} \label{refinement_cyclic}
If $V$ is a exponentiation-stable S-module, then so is $R(V)$.
\end{lemma}
\begin{proof}
Write $\Basic(V) = \{T_1, \dots, T_r\}$ and let $T \in \Basic(R(V))$. This is of the form $$T = T_\ell\, \cap \, \bigcap_{1\leq i, j \leq r} \{g \in G \mid \text{the coefficient of } \underline{g} \text{ in } \underline{T_i}\cdot \underline{T_j} \text{ is } \lambda_{i,j}\}$$ for some $\lambda_{i,j} \in \mathbb{Q}$. Then we have for $m \in \mathbb{N}$ with $m$ coprime to $|G|$:
\begin{align*}
    T^{(m)} &=  T_\ell^{(m)}\, \cap \, \bigcap_{1\leq i, j \leq r} \{g^m \in G \mid \text{the coefficient of } \underline{g} \text{ in } \underline{T_i}\cdot \underline{T_j} \text{ is }  \lambda_{i,j}\} \\
    &=  T_\ell^{(m)}\, \cap \, \bigcap_{1\leq i, j \leq r} \{g^m \in G \mid \text{the coefficient of } \underline{g^n} \text{ in } (\underline{T_i}\cdot \underline{T_j})^{(m)} \text{ is }  \lambda_{i,j}\} \\
    &= T_\ell^{(m)}\, \cap \, \bigcap_{1\leq i, j \leq r} \{g \in G \mid \text{the coefficient of } \underline{g} \text{ in } (\underline{T_i}\cdot \underline{T_j})^{(m)} \text{ is }  \lambda_{i,j}\} \\
    &= T_\ell^{(m)}\, \cap \, \bigcap_{1\leq i, j \leq r} \{g \in G \mid \text{the coefficient of } \underline{g} \text{ in } \underline{T_i}^{(m)}\cdot \underline{T_j}^{(m)} \text{ is }  \lambda_{i,j}\} \\
\end{align*}
where we used that $g_1^mg_2^m = (g_1g_2)^m$ for the fourth equality ($G$ is abelian) and repeatedly that $g \mapsto g^m$ is bijective. As $V$ is exponentiation stable we also have $\underline{T_i}^{(n)} \in V$ for all $1 \leq i \leq r$. Lemma \ref{refinement_product} implies that also $\underline{T_i}^{(n)} \cdot \underline{T_j}^{(n)} \in R(V)$ and we know $\underline{T_\ell}^{(n)} \in V \subseteq R(V)$. Thus Lemma \ref{S_module_extract} allows us to extract $\underline{T}^{(n)}$ from the $\underline{T_\ell}^{(n)}$ and $\underline{T_1}^{(n)} \cdot \underline{T_1}^{(n)}, \underline{T_1}^{(n)} \cdot \underline{T_2}^{(n)}, \dots ,\underline{T_r}^{(n)} \cdot \underline{T_r}^{(n)}$ with $$\phi(x,x_{1,1}, x_{1,2}, \dots , x_{r,r}) = \begin{cases}
1 & x = 1 \wedge \bigwedge_{1\leq i, j \leq r} x_{i,j} = \lambda_{i,j} \\
0 & \text{otherwise} 
\end{cases}$$
which proves $\underline{T}^{(n)} \in R(V)$. As this holds for every basic set of $R(V)$ it clearly also holds for every union of basic sets.
\end{proof}
As a result we obtain our main theorem of this section:
\begin{definition} \label{equiv_power}
Consider the equivalence relation on G given by $g_1 \sim g_2$ if and only if there is some $n \in \mathbb{N}$ with $n$ coprime to $|G|$ such that $g_1^n = g_2$. We define $d(G)$ to be the amount of equivalence classes of this equivalence relation.
\end{definition}
\begin{theorem} \label{wl_steps_cyclic_bound}
The 2-WL algorithm on a Cayley graph of a finite abelian group $G$ takes at most $(2+d(G)) \lceil \log(|G|) \rceil = \mathcal{O}(\log(|G|)d(G))$ steps to stabilize.
\end{theorem}
\begin{proof}
By Proposition \ref{equivalence_wl_refinement} it is enough to prove that any S-module $V$ stabilizes after at most $2\lceil \log(|G|) \rceil + d(G) \lceil \log(|G|) \rceil$ refinement steps. \\
Let $\Basic(V) = \{T_1, \dots, T_r\}$. By Lemma \ref{cylic_log_n} we know $\underline{T_i^{(m)}}= \underline{T_i}^{(m)} \in R^r(V)$ for all $1 \leq m \leq |G|$ which are coprime to $|G|$ if we take $r = 2\lceil\log(G) \rceil$.\\
Consider the S-module $V'$ associated to the partition $\bigwedge_{1 \leq m \leq |G| \wedge \gcd(m, |G|)=1} \{T^{(m)} \mid T \in \Basic(V)\}.$\\
Then $V'$ is exponentiation-stable: Every basic set $T$ of $V'$ is an intersection of some $T_{i_m}^{(m)}$, and hence $T^{(m')}$ is the intersection of $(T_{i_m}^{(m)})^{(m')} = T_{i_m}^{(m\cdot m')} = T_{i_m}^{(m \cdot m' \text{ mod } |G|)}$ which shows that also $T^{(m')}$ is a basic set of $V'$.\\
We also have $V \subseteq V'$ and $V' \subseteq R^r(V)$ by Corollary \ref{s_module_intersections} and hence if we have $R^\ell(V') = R^{\ell+1}(V')$ for some $\ell \in \mathbb{N}$, we also have $R^{r + \ell + 1}(V) \subseteq R^{r+ \ell+1}(V') = R^{\ell}(V') \subseteq R^{\ell}(R^r(V)) = R^{\ell + r}(V)$ (where we repeatedly used Lemma \ref{subset_refinement}) and thus $R^{r+\ell+1}(V) = R^{\ell + r}(V)$. \\
Therefore it remains to prove that the refinement on $V'$ stabilizes in at most $d(G) \lceil \log(|G|) \rceil$ steps.

For any S-module $W$ we define the map $\delta_W : G \rightarrow \mathbb{N}$: $g \mapsto |T_{(g)}|$ where we associate to every element the size of the basic set that contains it. \\
If $W$ is exponentiation-stable we have for $g \in G$ and $n \in \mathbb{N}$ with $n$ coprime to $|G|$ that $g^n \in T_{(g)}^{(n)}$ and thus $\delta_W(g^n) \leq \delta_W(g)$. If $m \in \mathbb{N}$ is the multiplicative inverse of $n$ modulo $|G|$ we have $(g^n)^m = g$ and hence we get $\delta_W(g^n) = \delta_W(g)$. This means that $\delta_W$ is determined by $d(G)$ many elements, coming from the equivalence classes of the equivalence relation of Definition \ref{equiv_power}. \\
If $R(W)$ is not the same as $W$, then some basic set has to change, i.e.\ there is some basic set $T \in \Basic(W)$ such that this is now a union of at least two non empty basic sets in $R(W)$: $T = T_1' \cup \dots \cup T_h'$ with $h \geq 2$, $T_i' \in \Basic(R(W))$. But then at least one of the $T_i'$ has to have size at most $\frac{|T|}{2}$ and for any $g \in T_i'$ with $|T_i'| \leq |T| / 2$ we have $\delta_{R(W)}(g) = |T_i'| \leq |T| / 2 = \delta_{W}(g) / 2$.\\
Lemma \ref{refinement_cyclic} proves that if $W$ is exponentiation-stable also $R(W)$ is exponentiation-stable. But if we fix a set $\{g_1, \dots, g_{d(G)}\}$ of representatives for the equivalence relation $\sim$, then if the refinement does not stabilize $W$ we have $\delta_{R(W)}(g_i) \leq  \delta_{W}(g_i) / 2$ for some $1 \leq i \leq d(G)$. Finally each $\delta_W(g)$ is at most $|G|$ and can thus be halved at most $\lceil \log(|G|) \rceil$ times.

Applied to $V'$ this proves that we have $R^\ell(V') = R^{\ell+1}(V')$ for some $\ell$ with $\ell \leq d(G)\cdot \lceil \log(|G|) \rceil$ as desired.
\end{proof}
For $G = \mathbb{Z}_n$ we get that $m_1, m_2 \in \mathbb{Z}_n$ are equivalent in the sense of Definition \ref{equivalence_wl_refinement} if and only if $\gcd(m_1, n) = \gcd(m_2, n)$. Hence $d(\mathbb{Z}_n)$ is the amount of divisors of $n$ and Theorem \ref{dnlogn_bound} is proved. For prime numbers $n$ this implies a bound of $\mathcal{O}(\log(n))$ steps and for prime powers $n=p^\ell$ a bound of $\mathcal{O}(\log(n)^2)$ steps. 
\section{Color Refinement on circulant graphs} \label{4}
\subsection{General constructions}
In this section we interpret the color refinement algorithm on vertex-colored Cayley graphs of finite groups in terms of operations on S-modules. For this $G$ is still an arbitrary finite group but we pick $K = \mathbb{C}$ (for this section it is again important to choose $K$ of characteristic $0$, but for what we are going to do we also need more algebraic elements, e.g.\ the $p$-th roots of unity). 
The central notion in the context of the color refinement algorithm is that of a $\Con$-refinement for a given connection set $\Con \subseteq G$, which is going to correspond to a single refinement step of the Color Refinement algorithm on the Cayley graph $\Cay(G, \Con)$:
\begin{definition}For any element $v = \sum_{g \in G}\lambda_g \underline{g} \in \mathbb{C}[G]$ we define the induced partition on $G$ given by the elements of $G$ with identical coefficients $\mathcal{C}_v := \{\{g \in G \mid \lambda_g = \lambda\} \mid \lambda \in \mathbb{C}\}$.\\
For $\Con \subseteq G$ and $V$ an S-module with basic sets $T_1, \dots, T_r$ we define the \high{$\Con$-refinement} of $V$ to be the S-module $R_\Con(V)$ corresponding to the partition $\Basic(V) \wedge \bigwedge_{1 \leq i \leq r} \mathcal{C}_{\underline{\Con} \cdot \underline{T_i}}$.
\end{definition}
\begin{lemma} \label{S_refinement_product}
Let $\Con \subseteq G$, $V$ an S-module and $v\in V$. Then $\underline{\Con} \cdot v$ (the product in $\mathbb{\mathbb{C}}[G]$) is an element of $R_\Con(V)$.
\end{lemma}
\begin{proof}
Let $\Basic(V) = \{T_1, \dots, T_r\}$. By definition $\underline{\Con} \cdot \underline{T_j} \in R_\Con(V)$ for every $1 \leq i,j \leq r$. The $\underline{T_i}$ form a basis of $V$ so we can write $v = \sum_{1\leq i \leq r}\mu_i \underline{T_i}$ with $\mu_i \in \mathbb{C}$. Then $\underline{\Con}\cdot v = \sum_{1\leq i \leq r} \mu_i (\underline{\Con}\cdot \underline{T_i}) \in R_\Con(V)$ as $R_\Con(V)$ is a vector space.
\end{proof}
\begin{lemma} \label{S_subset_refinement}
Let $\Con \subseteq G$ and let $V, W$ be S-modules such that $V \subseteq W$. Then also $R_\Con(V) \subseteq R_\Con(W)$.
\end{lemma}
\begin{proof}
Let $\Basic(V) = \{T_1, \dots, T_r\}$. By definition an element $T \in \Basic(R_\Con(V))$ is of the form $T = T_\ell \cap \bigcap_{1\leq i \leq r} T_{i}'$ for some $1\leq \ell \leq r$ and where the $T_{i}'$ are the elements of $G$ which appear in $\underline{\Con}\cdot \underline{T_i}$ with coefficient $\lambda_{i} \in \mathbb{C}$. We have $\underline{T_i} \in V \subseteq W$ and thus also $\underline{\Con}\cdot \underline{T_i} \in R_\Con(W)$ by Lemma \ref{S_refinement_product}. Applying Lemma \ref{S_module_extract} in $R_\Con(W)$ to $\underline{T_\ell}, \underline{\Con} \cdot \underline{T_{1}}, \dots  ,\underline{\Con}\cdot \underline{T_{r}}$ with $$\phi(x,x_1, \dots, x_r) = \begin{cases}
1 & x = 1 \wedge \bigwedge_{1\leq i \leq r} x_{i} = \lambda_{i} \\
0 & \text{otherwise} 
\end{cases}$$ implies $\underline{T} \in \Basic(R_\Con(W))$. Hence $R_\Con(V) \subseteq R_\Con(W)$.
\end{proof}
\begin{definition} Let $\Con \subseteq G$ be the connection set of a Cayley graph $\Cay(G, \Con)$.
Let $\mathcal{C}$ be a partition of $G$, which we interpret as a coloring of the vertices. For $g \in G$ we consider the multiset of colors of all in-neighbors (i.e.\ all vertices from which there exists an edge to $g$): 
$$L(g) := \{\{\mathcal{C}(s^{-1}g) \mid s \in \Con\}\}$$ 
Then we define one step of the \high{Color-Refinement algorithm} on the coloring $\mathcal{C}$ of $\Cay(G,\Con)$ to be the partition $\CR_{\Con}(\mathcal{C})$ induced by splitting elements with different labels $L$. The sets of $\CR_{\Con}(\mathcal{C})$ are then given by $\{g' \in G \mid \mathcal{C}(g) = \mathcal{C}(g') \wedge L(g') = L(g)\}$ for $g \in G$.
\end{definition}
\begin{lemma} \label{CR_automorphism_stable}
Let $\Cay(G,\Con)$ be a Cayley graph with $\Con \subseteq G$ and $\mathcal{C}$ a partition of $G$. If $\varphi : G \rightarrow G$ is an automorphism of $\Cay(G, \Con)$ such that $\mathcal{C}(g) = \mathcal{C}(\varphi(g))$ for all $g \in G$, then also $\CR_{\Con}(\mathcal{C})(g) = \CR_{\Con}(\mathcal{C})(\varphi(g))$, i.e.\ the color refinement algorithm respects automorphisms of the colored graph.
\end{lemma}
\begin{proof}
As $\varphi$ is a graph automorphism it respects the neighbors of vertices. In particular $\{\varphi(\con^{-1}g) \mid \con \in \Con\} = \{\con^{-1}\varphi(g) \mid \con \in \Con\}$ for all $g \in G$. Together with $\mathcal{C}(\con^{-1}g) = \mathcal{C}(\varphi(\con^{-1}g))$ this implies $L(g) = L(\varphi(g))$ for all $g\in G$ and hence also $\CR_{\Con}(\mathcal{C})(g) = \CR_{\Con}(\mathcal{C})(\varphi(g))$.
\end{proof}
\begin{definition}
To every partition $\mathcal{C} = \{T_1, \dots, T_n\}$ of $G$ we associate the \high{induced S-module} $S(\mathcal{C})$ whose basic sets are the $T_i$.
\end{definition}
\begin{lemma}
\label{cr_refinement} Let $\Con \subseteq G$
and let $\mathcal{C}$ be a partition of $G$. Then $S(\CR_{\Con}(\mathcal{C})) = R_\Con(S(\mathcal{C}))$ i.e.\ the $\Con$-refinement of the induced S-module is precisely the induced S-module after executing one step of the Color Refinement algorithm.
\end{lemma}
\begin{proof}
Let $\Basic(S(C)) = \mathcal{C} = \{C_1, \dots, C_r\}$.  For the amount of in-neighbors of $g$ in $C_i$ (given by $|U^{(-1)}g\cap C_i|$) we get
$$|\Con^{(-1)}g \cap C_i| = |\{\con \in \Con, c \in C_i \mid \con c = g\}| = \text{coefficient of } g \text{ in } \underline{\Con}\cdot \underline{C_i}$$
A basic set of $S(\CR_U(\mathcal{C}))$ consists of the elements from a $C_j$ with fixed amount of neighbors in $C_i$ for all $1 \leq i \leq r$ (and hence equivalently with fixed coefficients in the $\underline{\Con} \cdot \underline{C_i}$). This immediately corresponds to a basic set of $R_\Con(S(\mathcal{C}))$. 
\end{proof}
 Suppose we have a partition $\mathcal{C}$ (a coloring) of the Cayley graph $\Cay(G, \Con)$ with $\Con \subseteq G$. We write $\CR_{\Con}^*(\mathcal{C})$ for the partition we obtain by applying the Color-Refinement algorithm to $\mathcal{C}$ until it stabilizes. Then Lemma \ref{cr_refinement} (using Lemma \ref{S_refinement_product} and Lemma \ref{S_subset_refinement}) implies that: 
\begin{proposition} \label{cr_by_multiplication} $S(\CR_{\Con}^*(\mathcal{C}))$ is precisely the smallest S-module that contains $S(\mathcal{C})$ and is stable under multiplication by $\underline{\Con}$. 
\end{proposition}
\subsection{Tinhofer property of \texorpdfstring{$\mathbb{Z}_p$}{Zp}}
In this section we use the interpretation of color refinement using S-modules to prove that all Cayley graphs over $\mathbb{Z}_p$ have the Tinhofer property. Our main tool for this are going to be the eigenvalues and eigenspaces of the linear map $(\underline{\Con}\cdot) : \mathbb{C}[G] \rightarrow \mathbb{C}[G]$ given by multiplication with $\underline{\Con}$ for $\Con \subseteq G$ the connection set of our graph, i.e.\ the spectrum of the adjacency matrix of the corresponding graph.

From now on we fix $p$ a prime and only consider $G = \mathbb{Z}_p = \{0, \dots, p-1\}$ (using additive notation, so that $g_1g_2$ for $g_1, g_2 \in \mathbb{Z}_p$ denotes the multiplication modulo $p$). If we fix some primitive $p$-th root of unity $\xi_p \in \mathbb{C}$ we get the dual basis of $\mathbb{C}[\mathbb{Z}_p]$ coming from the representation theory of $\mathbb{Z}_p$, i.e.\ from the characters $G \rightarrow \mathbb{C}$, $g \mapsto (\xi_p)^{gk}$:
$$e_k := \frac{1}{p}\sum_{g \in \mathbb{Z}_p} \xi_p^{g k}\underline{g}$$
for $0 \leq k \leq p-1$. For $g \in \mathbb{Z}_p$ we get the expression:
\begin{align*}
    \underline{g} = \frac{1}{p}\sum_{g' \in \mathbb{Z}_p}\sum_{k = 0}^{p-1} \xi_p^{(g-g')k} \underline{g'}= \sum_{k = 0}^{p-1} \xi_p^{gk} e_k 
\end{align*}
where the first equality uses $\sum_{g=0}^{p-1}\xi_p^g = 0$. In particular (by the orthogonality relations of the characters) this is a basis of idempotents:
\begin{align*}
    e_k \cdot e_\ell &= \frac{1}{p^2}\sum_{g' \in \mathbb{Z}_p}\sum_{g \in \mathbb{Z}_p} \xi_p^{g k}\xi_p^{g'\ell}\underline{g+ g'} = \frac{1}{p^2} \sum_{g \in \mathbb{Z}_p} \left(\sum_{\substack{g_1,g_2 \in \mathbb{Z}_p \\ g_1+g_2 = g}} \xi_p^{g_1k+g_2\ell}  \right) \underline{g} \\ &= \frac{1}{p^2} \sum_{g \in \mathbb{Z}_p} \left(\sum_{g_1 \in \mathbb{Z}_p} \xi_p^{g_1(k-\ell) + g\ell} \right) \underline{g} = \begin{cases}
e_k & k = \ell \\
0 & k \neq \ell
\end{cases}
\end{align*}
where we used $\sum_{g=0}^{p-1}\xi_p^g = 0$ for the last equality again. 

For a fixed $\Con \subseteq \mathbb{Z}_p$ we can then compute $\underline{\Con} = \sum_{\con \in \Con} \underline{\con} = \sum_{k =0}^{p-1} \left(\sum_{\con \in \Con} \xi_p^{\con k}\right) e_k$ which implies that the linear map $(\underline{\Con}\cdot) : \mathbb{C}[\mathbb{Z}_p] \rightarrow \mathbb{C}[\mathbb{Z}_p]$ has eigenvectors $e_k$ for $0 \leq k \leq p-1$ with eigenvalues $\lambda_k := \sum_{\con \in \Con} \xi_p^{\con k}$ respectively. To analyse the eigenspaces we determine which eigenvalues are the same:
\begin{lemma} \label{eigenvalues_same}
For $\emptyset \neq \Con \subseteq \mathbb{Z}_p \setminus \{0\}$ the eigenvalues $\lambda_k = \sum_{\con \in \Con} \xi_p^{\con k}$ are equal for $k$ and $k'$ if and only if there is some $\ell \in \mathbb{Z}_p \setminus \{0\}$ with $k = \ell k'$ and $\{ \con \ell \mid \con \in \Con\} = \Con$.
\end{lemma}
\begin{proof}
As $0 \not \in \Con$ we have $\lambda_k = |\Con|$ if and only if $k = 0$. In particular the statement holds for $k=0$ or $k'=0$ and we can assume $k \neq 0$ and $k'\neq 0$.

If there exists such an $1 \leq \ell \leq p-1$ with $\{ \con \ell \mid \con \in \Con\} = \Con$ and $k = \ell k'$ we compute:
$$\lambda_k = \sum_{\con \in \Con} \xi_p^{\con k} = \sum_{\con \in \Con} \xi_p^{\con \ell k'} = \sum_{\con \in \Con} \xi_p^{\con k'} = \lambda_{k'}$$
On the other hand, if $\lambda_k = \lambda_{k'}$ we have $\sum_{\con \in \Con} \xi_p^{\con k} = \sum_{\con \in \Con} \xi_p^{\con k'}$. But in both sums only $\xi_p^i$ for $1 \leq i \leq p-1$ appear, which are linearly independent over $\mathbb{Q}$. Then this equality is a linear dependence, which implies that in both sums all exponents have to appear equally often. In particular, as every exponent appears at most once, we have $\{\con k \mid \con \in \Con\} = \{\con k' \mid \con \in \Con\}$, i.e.\ $\{\con (kk'^{-1}) \mid \con \in \Con\} = \Con$ as desired.
\end{proof}
\begin{definition}
For $\emptyset \neq \Con \subseteq \mathbb{Z}_p \setminus\{0\}$ we define the subgroup of $\mathbb{Z}_p^{\times}$ that fixes $\Con$ as $H_\Con := \{h \in \mathbb{Z}_p^{\times} \mid \{\con h \mid \con \in \Con\} = \Con\}$. We write $d_\Con := [\mathbb{Z}_p^\times : H_\Con] = (p-1) / |H_\Con|$ for the index of this subgroup.
\end{definition}
\begin{remark} \label{Z_p_automorphisms}
For $\emptyset \neq \Con \subseteq \mathbb{Z}_p\setminus \{0\}$ the previous lemma now states that the eigenvalues $\lambda_k$ have equivalence classes $\{0\}$ and the cosets of $H_\Con$, i.e.\ $\{g H_\Con \mid 0 \leq g\leq p-1 \}$. \\
Note that for any $h \in H_\Con$ and $b \in \mathbb{Z}_p$ the linear map $\varphi_{h,b} : \mathbb{Z}_p \rightarrow \mathbb{Z}_p: g \mapsto gh+b$ is an automorphism of the Cayley graph $\Cay(\mathbb{Z}_p, \Con)$, as we have $\varphi_{h,b}(\con+g) = (\con+g)h+b = \con h + (gh+b) = \con h + \varphi_{h,b}(g)$ and $\con h \in \Con$, i.e.\ all neighbors of $g$ are mapped to neighbors of $\varphi_{h,b}(g)$. 
\end{remark}
Now we can calculate what happens to the coloring of our Cayley graph after individualizing one vertex:
\begin{lemma} \label{Z_p_one}
Let $\emptyset \neq \Con \subseteq \mathbb{Z}_p \setminus \{0\}$ and let $g_0 \in \mathbb{Z}_p$. Let $\mathcal{C} = \{\{g_0\}, \mathbb{Z}_p \setminus \{g_0\}\}$ be the coloring where we individualize the vertex $g_0$. Then $\CR_\Con^*(\mathcal{C}) = \{g_0 + aH_\Con \mid a \in \mathbb{Z}_p\}$.
\end{lemma}
\begin{proof}
By Proposition \ref{cr_by_multiplication} we have that $V := S(\CR_{\Con}^*(\mathcal{C}))$ is the smallest S-module that contains $S(\mathcal{C})$ and that is stable under multiplication by $\underline{\Con}$.

The map $\varphi_{1, g_0}: g \mapsto g + g_0$ is an isomorphism of colored graphs, where $0$ and $g_0$ are individualized respectively. Hence we can assume without loss of generality that $g_0 = 0$.

We will use the following fact from linear algebra:\\
\textit{Let $W' \subseteq W$ be  finite dimensional vector spaces and $\psi : W \rightarrow W$ a diagonalizable linear endomorphism, such that $\psi(W') \subseteq W'$ with $v_1, \dots, v_n \in W$ a basis of $W$ of eigenvectors corresponding to the eigenvalues $\mu_1, \dots, \mu_n$. Then if $v = \sum_{i=1}^{n} a_i v_i \in W'$ we also have $\sum_{1 \leq i \leq n \,\mu_i = \mu} a_iv_i \in W'$  for any $\mu\in \mathbb{C}$, i.e.\ the projection to the eigenspaces are in $W'$.}

In our case we have that the $e_k$ are a basis of eigenvectors of the linear endomorphism $\underline{\Con}\cdot$ and Lemma \ref{eigenvalues_same} says that $\lambda_i = \lambda_j$ precisely for $i, j$ with $i = \con \cdot j$ and $\con \in H_\Con$. In particular we get $d_\Con + 1$ different eigenvalues. Applying the fact to $\underline{g_0} = \underline{0} = \sum_{k = 0}^{p-1} e_k \in V$ we hence get $d_\Con+1$ non trivial elements in $V$ which consist of disjoint basis vectors and thus $\dim_\mathbb{C} V \geq d_\Con + 1$.

On the other hand, we have automorphisms $\varphi_{h, 0} : \mathbb{Z}_p \rightarrow \mathbb{Z}_p$ for $h \in H_\Con$, which map $0$ to $0$, hence respect the partition $\mathcal{C}$ and hence also $\CR_{\Con}^*(\mathcal{C})$ by Lemma \ref{CR_automorphism_stable}. Then $g \in \mathbb{Z}_p$ has to be in the same basic set of $V$ as $hg$ for all $h \in H_\Con$. For a fixed $g \in \mathbb{Z}_p\setminus \{0\}$ the $hg$ for $h \in H_\Con$ are pairwise distinct and hence $g$ is contained in a basic set of size at least $|H_\Con|$. But then we can have at most $(p-1)/|H_\Con| = d_\Con$ basic sets to cover these $p-1$ elements. In total this says $\dim_\mathbb{C}V \leq d_\Con+1$.

All in all this shows $\dim_{\mathbb{C}}V = d_\Con+1$ and we can only have equality for the upper bound if every element $g\in \mathbb{Z}_p$ with $g \neq 0$ is in a basic set of $\CR_{\Con}^*(\mathcal{C})$ with $|H_\Con|$ elements given by $\{gh \mid h \in H_\Con\} = g H_\Con$ as desired.
\end{proof}

Finally we can calculate what happens to the coloring of our Cayley graph after individualizing two vertices:
\begin{lemma} \label{Z_p_two}
Let $\emptyset \neq \Con \subseteq \mathbb{Z}_p \setminus \{0\}$ with $\Con \neq \mathbb{Z}_p \setminus \{0\}$ and let $g_0,g_1 \in \mathbb{Z}_p$ with $g_0 \neq g_1$. Let $\mathcal{C} = \{\{g_0\}, \{g_1\}, \mathbb{Z}_p \setminus \{g_0,g_1\}\}$ be the coloring where we individualize the vertices $g_0$ and $g_1$. Then $\CR_{\Con}^*(\mathcal{C}) = \{\{g\} \mid g \in \mathbb{Z}_p\}$, i.e.\ every vertex gets a unique color.
\end{lemma}
\begin{proof}
By Proposition \ref{cr_by_multiplication} we have that $V := S(\CR_\Con^*(\mathcal{C}))$ is the smallest S-module that contains $S(\mathcal{C})$ and that is stable under multiplication by $\underline{\Con}$.

Let $\lambda_0', \dots, \lambda_{d_\Con}'$ (with $\lambda_0' = |\Con|$) be all the distinct eigenvalues of the multiplication by $\underline{\Con}$. By the fact in the proof of Lemma \ref{Z_p_one} we have that $V = E_0 \oplus E_1 \oplus \dots \oplus E_{d_\Con} \subseteq \mathbb{C}[\mathbb{Z}_p]$, where $E_j$ is the space of eigenvectors of the multiplication by $\underline{\Con}$ corresponding to the eigenvector $\lambda_j'$.

We have $\dim_\mathbb{C} E_0 = 1$ because $e_0 = \frac{1}{p}\sum_{g \in \mathbb{Z}_p} \underline{g} \in V$ (and there is only one dimension of eigenvectors for this eigenvalue in $\mathbb{C}[\mathbb{Z}_p]$). We also claim $\dim_{\mathbb{C}} E_i = \dim_{\mathbb{C}} E_j$ for all $1 \leq i, j \leq d_\Con$:

Let $1 \leq \ell \leq p-1$. From the theory of cyclotomic fields (see e.g.\ \cite[Chapter 6.3]{lang2012algebra}) we know that there is a field automorphism of $\mathbb{Q}(\xi_p)$ which maps $\xi_p$ to $\xi_p^{\ell}$. Then we can lift this automorphism to a field automorphism $\sigma_\ell : \mathbb{C}\rightarrow \mathbb{C}$ with $\sigma_\ell(\xi_p) = \xi_p^{\ell}$ using Galois theory (e.g.\ \cite[Theorem 7]{C_automorphisms} or we could have just worked with $K = \mathbb{Q}(\xi_p)$ instead of $\mathbb{C}$ in our case of $G=\mathbb{Z}_p$). Now consider the $\sigma_\ell$-semilinear map $\chi_{\ell} : \mathbb{C}[\mathbb{Z}_p] \rightarrow \mathbb{C}[\mathbb{Z}_p]$ induced by $\mu e_i \mapsto \sigma_\ell(\mu) e_{\ell\cdot i}$. We compute for any $g \in \mathbb{Z}_p$:
$$\chi_{\ell}(\underline{g}) = \chi_{\ell}(\sum_{k = 0}^{p-1} \xi_p^{kg} e_k) = \sum_{k = 0}^{p-1} \sigma_\ell(\xi_p^{kg}) e_{\ell k} = \sum_{k = 0}^{p-1} \xi_p^{\ell kg} e_{\ell k}  = \underline{g}$$
In particular we get $\chi_{\ell}(\underline{T}) = \underline{T}$ for any $T \subseteq \mathbb{Z}_p$. Any $v \in V$ is of the form $\mu_1 \underline{T_1} + \dots + \mu_r \underline{T_r}$ for some $\mu_1, \dots, \mu_r \in \mathbb{C}$ and $\underline{T_1}, \dots, \underline{T_r} \in V$, which implies $\chi_{\ell}(v) = \sigma_\ell(\mu_1)\underline{T_1} + \dots + \sigma_\ell(\mu_r)\underline{T_r} \in V$, i.e.\ the set $V$ is fixed under $\chi_{\ell}$. \\
Now let $1 \leq i, j \leq d_\Con$ and $\dim_{\mathbb{C}}E_i = d$ and hence $v_1, \dots, v_d \in V$ linear independent elements which are eigenvectors for the eigenvalue $\lambda_i'$. By Lemma \ref{eigenvalues_same} the eigenvectors for the eigenvalue $\lambda_i'$ are linear combinations of $e_k$ for $k \in h_iH_\Con$ for some $h_i \in \mathbb{Z}_p\setminus\{0\}$. Similarly the eigenvectors for the eigenvalue $\lambda_j'$ are linear combinations of $e_k$ for $k \in h_jH_\Con$ for some $h_j \in \mathbb{Z}_p\setminus\{0\}$. If we let $\ell := h_i^{-1}h_j$ we get that $\chi_\ell(e_k)$ is an eigenvector for the eigenvalue $\lambda_j'$ if $e_k$ is an eigenvector for the eigenvalue $\lambda_i'$. Hence $\chi_\ell(v_1), \dots, \chi_\ell(v_d) \in V$ are eigenvectors for the eigenvalue $\lambda_j'$. But they are still linearly independent because $\mu_1\chi_\ell(v_1) + \dots + \mu_d\chi_\ell(v_d)=0$ implies $\sigma_\ell^{-1}(\mu_1) v_1 + \dots + \sigma_\ell^{-1}(\mu_d) v_d = 0$ (because $\chi_\ell$ is injective) i.e.\ $\mu_1 = \dots = \mu_d = 0$. This proves $\dim_{\mathbb{C}}E_j \geq d = \dim_{\mathbb{C}}E_i$ and by symmetry the desired equality.

Write $d := \dim_{\mathbb{C}} E_1$. Then we get $\dim_\mathbb{C}V = 1 + d_\Con \cdot d$. We have to prove $d = |H_\Con|$ as this implies $\dim_\mathbb{C} V = 1 + d_\Con \cdot |H_\Con| = p$ and hence $V = \mathbb{C}[\mathbb{Z}_p]$.

Now suppose for the sake of contradiction that $d < |H_\Con|$. In particular we have $|\Basic(V)| = 1 + d_\Con \cdot d$. Now consider the $d+1$ smallest (non empty) sets of $\Basic(V)$ given by $T_1, \dots, T_{d+1}$. We claim that $|T_1| + \dots + |T_{d+1}| \leq |H_\Con|$:

We get $|H_\Con| \neq 1$, as we have $|H_\Con| > d \geq 1$ because there is an eigenvector to every eigenvalue $\lambda_i$ in $V$ as seen in the proof of Lemma \ref{Z_p_one}. 
If $d=1$ the claim is immediately true, as there are two sets of size $1$, given by $\{g_0\}$ and $\{g_1\}$.

Now suppose $d > 1$.
As $\Con \neq \mathbb{Z}_p \setminus \{0\}$ we get $H_\Con \neq \mathbb{Z}_p \setminus \{0\}$ and hence $d_\Con > 1$. Suppose for the sake of contradiction that $|T_1| + \dots + |T_{d+1}| > |H_\Con|$. 
If we remove the two sets of size one (given by $\{g_0\}$ and $\{g_1\}$), then the $d-1$ remaining sets have total size at least $|H_\Con| - 1$. As the $T_i$ are the smallest basic sets, this implies that every other set must have size at least $\frac{|H_\Con|-1}{d-1}$.
All in all this shows that the $d_\Con d+1$ sets in $\Basic(U)$ must have total size at least 
$|H_\Con| + 1 +  (d_\Con d+1 - (d+1))\cdot \frac{|H_\Con|-1}{d-1} = 2 + (d_\Con d-1) \frac{|H_\Con|-1}{d-1}$. 

We calculate when this is greater than $p$:
\begin{align*}
   && 2 + (d_\Con d-1)\cdot \frac{|H_\Con|-1}{d-1} &> p \\
    \Leftrightarrow && (d_\Con d-1)\frac{(p-1)-d_\Con}{d_\Con (d-1)} + (2-p)&> 0 \\
    \Leftrightarrow && (d_\Con d-1)((p-1)-d_\Con) + (2-p)d_\Con (d-1) &> 0 \\
    \Leftrightarrow && -d_\Con^2d + d_\Con d(p-1) - (p-1) +d_\Con + (2-p)d_\Con d -(2-p)d_\Con  &> 0 \\
    \Leftrightarrow && -d_\Con^2d + d_\Con d - (p-1) + (p-1)d_\Con  &> 0 \\
    \Leftrightarrow && (d_\Con -1)( (p-1)-d_\Con d)  &> 0 \\
    \Leftrightarrow && (p-1)-d_\Con d  &> 0 \\
    \Leftrightarrow &&  d  &< \frac{p-1}{d_\Con} = |H_\Con| 
\end{align*} 
where we used $d > 1$ and $d_\Con > 1$. As $d < |H_\Con|$ this shows that the total sum is always greater than $p$ which is a contradiction as the sum of the sizes of all basic sets has to be exactly $|G| = |\mathbb{Z}_p| = p$ and hence we get $|T_1| + \dots + |T_{d+1}| \leq |H_\Con|$ in every case as desired.

Finally we claim that the elements $\underline{T_1}, \dots, \underline{T_{d+1}}$ induce $d+1$ linearly independent elements in $E_1$, which would be a contradiction to $\dim_{\mathbb{C}}E_1 = d$.

Without loss of generality let $\lambda_1'$ correspond to the eigenvalue with eigenvectors $e_k$ for all $k \in H_\Con$ (by just renumbering the $\lambda_i'$ for $i > 0$). We can write $H_\Con = \{k_1, \dots, k_{r}\}$ for some $k_1 < \dots < k_{r}$ and $r = |H_\Con|$. If we apply the fact of the proof of Lemma \ref{Z_p_one} to a $\underline{T}$ for $T \subseteq \mathbb{Z}_p$ for the eigenvalue $\lambda_1'$ we get the element $\sum_{k \in H_\Con}\sum_{g \in T} \xi_p^{kg} e_k \in V$ and we want to show that these elements for the $T_1, \dots, T_{d+1}$ are linearly independent, which is equivalent to the matrix $(\sum_{g \in T_j} \xi_p^{k_ig})_{ 1 \leq i \leq r, 1\leq j\leq d+1}$ having full column rank (as the $e_{k_1}, \dots, e_{k_r}$ are linearly independent).

A classical theorem due to Chebotarëv \cite{Stevenhagen1996} states that every square submatrix of the Vandermonde matrix $(\xi_p^{ij})_{0 \leq i,j \leq p-1}$ is invertible. If we consider the set $T := \bigcup_{1 \leq i \leq d+1} T_i$ we get $|T| \leq |H_\Con|$ by what we proved before. Write $T = \{g_1', \dots, g_{r'}'\}$ with $g_1' < \dots < g_{r'}'$ and $r' = |T|$. Then the Chebotarëv Theorem implies that the square sub-matrix of the Vandermonde matrix given by $(\xi_p^{k_ig_j'})_{1\leq i,j \leq r'}$ is invertible and thus that the matrix $(\xi_p^{k_ig_j'})_{1\leq i \leq r, 1 \leq j \leq r'}$ has full column rank. Finally, from this we obtain our matrix $(\sum_{g \in T_j} \xi_p^{k_ig})_{ 1 \leq i \leq r, 1\leq j\leq d+1}$ by merging columns using addition, which then implies that the new columns are also linearly independent (as they all depend on disjoint subsets of the linearly independent columns) as desired.

This is a contradiction and we get $\dim_{\mathbb{C}} E_1 = |H_\Con|$, i.e.\ $\dim_{\mathbb{C}} V = p$.
\end{proof}
Lemmas \ref{Z_p_one} and \ref{Z_p_two} together prove Theorem \ref{tinhofer_prop}, that every Cayley graph over $\mathbb{Z}_p$ has the Tinhofer property:
\begin{proof}[Proof (Theorem \ref{tinhofer_prop})]
We only have to see that after every individualization we have enough automorphisms. The cases $\Con = \mathbb{Z}_p \setminus \{0\}$ and $\Con = \emptyset$ are clearly correct, as every permutation of vertices is an automorphism. Now let $\emptyset \neq \Con \neq \mathbb{Z}_p \setminus \{0\}$. As every Cayley graph is regular the first color refinement ends with every vertex having the same color, which is fine as every Cayley graph is vertex transitive. After individualizing one vertex and applying the color refinement, Lemma \ref{Z_p_one} tells us that we are done only if $H_\Con = \{1\}$. Otherwise we see that the automorphisms still act transitively on the color classes as discussed in the proof of the lemma. Finally if $H_\Con \neq \{1\}$, we have to individualize another vertex and Lemma \ref{Z_p_two} implies that every vertex has a unique color now. 
\end{proof}
\begin{remark}
Lemma \ref{Z_p_one} and Lemma \ref{Z_p_two} also imply the classical theorem that the automorphisms of non-trivial circulant graphs of prime order are precisely the linear ones listed in Remark \ref{Z_p_automorphisms} (see \cite{Mueller2005} for a more general discussion on this):\\  If there was any other automorphism $\psi$ on such a non-trivial Cayley graph $\Cay(\mathbb{Z}_p, \Con)$ (i.e.\ $\Con \neq \mathbb{Z}_p\setminus \{0\}$ and $\Con \neq \emptyset$), also $\psi' : x \mapsto \psi(x) - \psi(0)$ would be an automorphism, which maps $0$ to $0$ and which is also not listed in Remark \ref{Z_p_automorphisms}. But Lemma \ref{Z_p_one} together with Lemma \ref{CR_automorphism_stable} shows that $\psi'(1) \in H_\Con$ and hence $\psi'' : x \mapsto \psi'(1)^{-1}\psi'(x)$ is also a non trivial automorphism (as otherwise $\psi'$ would have been listed in Remark \ref{Z_p_automorphisms}). But $\psi''$ fixes $0$ and $1$ and Lemma \ref{CR_automorphism_stable} tells us that after individualizing $0$ and $1$ not every vertex would have a unique color, which is a contradiction to Lemma \ref{Z_p_two}. 
\end{remark}
\subsection{A counterexample for \texorpdfstring{$\mathbb{Z}_4 \times \mathbb{Z}_4$}{Z4 x Z4}}
We give a counterexample showing that not every (undirected) Cayley graph of a finite (abelian) group satisfies the Tinhofer property:

Consider $X = \Cay(G, S)$ for $G = \mathbb{Z}_4 \times \mathbb{Z}_4$ and $S = \{(1,0), (3, 0), (0, 1), (0, 3), (1, 1), (3, 3)\}$. This is an undirected Cayley-graph of a finite abelian group. We compute the color refinement of $X$ after individualizing the vertex $(0,0) \in G$, by listing all color classes after each round of refinement: 
\begin{itemize}
    \item[0.] $\{(0, 0)\}$, $X \setminus \{(0, 0)\}$
    \item[1.] $\{(0, 0)\}$, $\{(1, 0), (3, 0), (0, 1), (0, 3), (1, 1), (3, 3)\}$, \\$\{(0, 2), (1, 2), (1,3), (2, 0),(2, 1), (2, 2), (2, 3), (3, 1), (3, 2)\}$
    \item[2.] $\{(0, 0)\}$, $\{(1, 0), (3, 0), (0, 1), (0, 3), (1, 1), (3, 3)\}$, $\{(0, 2), (1,3), (2, 0), (2, 2), (3, 1)\}$, \\$\{(1,2), (2,1), (2, 3), (3, 2)\}$
    \item[3.] $\{(0, 0)\}$, $\{(1, 0), (3, 0), (0, 1), (0, 3)\}$, $\{(1, 1), (3, 3)\}$, $\{(0, 2), (1,3), (2, 0), (3, 1)\}$, \\$\{(2, 2)\}$, $\{(1,2), (2,1), (2, 3), (3, 2)\}$
\end{itemize}
In particular, $\{(0,2), (1,3), (2,0), (3,1)\}$ forms one of the resulting color classes (and it is the only interesting color class, as the elements of any other color class are related by automorphisms of $X$ given by $(a,b) \mapsto (b,a)$ or $(a,b)\mapsto (-a, -b)$ or $(a,b) \mapsto (-b, -a)$ which all fix $(0,0)$). 

If $X$ would satisfy the Tinhofer property, there would be a graph automorphism $\varphi: G \rightarrow G$ which fixes all of the color classes above and sends $(1,3)$ to $(0,2)$. In particular, if $\varphi(A) = B$ for subsets $A, B \subseteq G$, then $\varphi$ maps the neighbors of $A$ in a particular color class to the neighbors of $B$ in that color class. Thus the neighbors of $\varphi((1,3)) = (0,2)$ in the sixth color class give $\varphi(\{(1,2),(2,3)\}) = \{(1,2), (3,2)\}$ and their neighbors in the second color class give $\varphi(\{(3,0), (0,1)\}) = \{(0,3),(0,1)\}$. But the neighbors of $\varphi((1,3)) = (0,2)$ in the second color class give $\varphi(\{(0,3), (1,0)\}) = \{(0,3), (0,1)\}$ and hence $\varphi(\{(3,0), (0,1), (0,3), (1,0)\}) = \{(0,3), (0,1)\}$. This is a contradiction to $\varphi$ being a bijection. Then there cannot exist such an automorphism and hence $X$ does not satisfy the Tinhofer property.
\bibliographystyle{plain}
\bibliography{main}
\end{document}